\documentclass[12pt]{article}
\usepackage[utf8]{inputenc}
\usepackage[T1]{fontenc}

\usepackage{amsmath, amsthm}
\usepackage{amssymb}

\newtheorem{theorem}{Theorem}[section]

\newtheorem{corollary}[theorem]{Corollary}
\newtheorem{lemma}[theorem]{Lemma}

\title{Upper bound for the number of privileged words}

\author{Josef Rukavicka\thanks{Department of Mathematics,
Faculty of Nuclear Sciences and Physical Engineering, CZECH TECHNICAL UNIVERSITY
IN PRAGUE
(josef.rukavicka@seznam.cz).}}

\newtheorem{definition}[theorem]{Definition}
\theoremstyle{remark}
\newtheorem{example}[theorem]{Example}
\newtheorem{remark}[theorem]{Remark}

\newcommand{\Factor}{F}
\newcommand{\Alphabet}{A}
\newcommand{\bound}{D}
\newcommand{\priv}{B}
\newcommand{\ktwo}{2}
\newcommand{\cca}{c_1}
\newcommand{\ccb}{c_2}
\newcommand{\ccc}{c_3}
\newcommand{\ccd}{c_4}
\newcommand{\cce}{c_5}
\newcommand{\ccf}{c_6}

\DeclareMathOperator{\up}{upb}

\date{\small{Mai 25, 2022}}

\begin{document}
\maketitle

\begin{abstract}
A non-empty word $w$ is a \emph{border} of a word $u$ if $\vert w\vert<\vert u\vert$ and $w$ is both a prefix and a suffix of $u$. A word $u$ is \emph{privileged} if $\vert u\vert\leq 1$ or if $u$ has a privileged border $w$ that appears exactly twice in $u$. 
Peltomäki (2016) presented the following open problem: ``Give a nontrivial upper bound for $B(n)$'', where $B(n)$ denotes the number of privileged words of length $n$. 

Let $\ln^{[0]}{(n)}=n$ and let $\ln^{[j]}{(n)}=\ln{(\ln^{[j-1]}{(n)})}$, where $j,n$ are positive integers.
We show that if $q>1$ is a size of the alphabet and $j\geq 3$ is an integer then there are constants $\alpha_j$ and $n_j$ such that \[B(n)\leq \alpha_j\frac{q^{n}\sqrt{\ln{n}}}{\sqrt{n}}\ln^{[j]}{(n)}\prod_{i=2}^{j-1}\sqrt{\ln^{[i]}(n)}\mbox{, where }n\geq n_j\mbox{.}\] This result improves the upper bound of Rukavicka (2020).
\end{abstract}

\section{Introduction}
Let $u,w$ be non-empty words. We say that $w$ is a \emph{border} of $u$ if $\vert w\vert<\vert u\vert$ and $w$ appears as both a prefix and a suffix of $u$.
Let \[\Theta(u)=\{w\mid w\mbox{ is a border of }u\}\mbox{.}\] We say that $w$ is the \emph{maximal border} of $u$  if for every $\overline w\in\Theta(u)$ we have that $\vert \overline w\vert\leq \vert w\vert$.  
We say that $u$ is \emph{closed} if there is $w\in\Theta(u)$ such that $u$ contains exactly two occurrences of $w$; realize that these two occurrences are a prefix and a suffix of $u$.  

We say that $u$ is \emph{privileged} if $\vert u\vert\leq 1$ or if there is $w\in\Theta(u)$ such that $w$ is privileged and $w$ appears exactly twice in $u$. It is clear that every privileged word $u$ with $\vert u\vert>1$ is also a closed word.

The closed and privileged words attracted some attention in recent years  \cite{KeLeSa2013}, \cite{Pelto2013}, \cite{ScSh16}. To find a lower and an upper bound for the number of privileged words are two topics that have been researched. 
Concerning the lower bound, it was shown that there are constants $c$ and $n_0$ such that for all $n>n_0$, there are at least $\frac{cq^n}{n(\log_q{n})^2}$ privileged words of length $n$ \cite{Nicholson2018_priviligedwords}, where $q$ denote the size of the alphabet in question. This improves the lower bound for the number of privileged words from \cite{FoJaPeSha2016}. 

Let $\priv(n)$ denote the number of privileged words of length $n$. As for an upper bound for the number of privileged words, the following open problem can be found  \cite{Pelto2016}: ``Give a nontrivial upper bound for $\priv(n)$''.

Let $\bound(n)$ denote the number of closed words of length $n$. 
In \cite{RUKAVICKA2020105917}, it was shown that if $q>1$ is a size of the alphabet then there is a positive real constant $c$ such that \begin{equation}\label{dyhc67vzcxk}\bound(n)\leq c\ln{n}\frac{q^{n}}{\sqrt{n}}\mbox{, where }n>1\mbox{.}\end{equation} Since every privileged word $u$ with $\vert u\vert>1$ is also a closed word, we have that (\ref{dyhc67vzcxk}) is also an upper bound for the number of privileged words. 
Hence, the upper bound (\ref{dyhc67vzcxk}) gave a response to the open problem in \cite{Pelto2016}.  
\begin{definition}
Let $\mathbb{N}$ denote the set of positive integers and let $\mathbb{R}$ denote the set of real numbers.
Let $\ln^{[0]}{(n)}=n$ and let $\ln^{[j]}{(n)}=\ln{(\ln^{[j-1]}{(n)})}$, where $j,n\in\mathbb{N}$.

Given $j\in\mathbb{N}$, let $\sigma^{[j]},\rho^{[j]}:\mathbb{N}\rightarrow\mathbb{R}$ be functions defined as follows: 
\[\sigma^{[1]}(n)=\sqrt{\ln{n}}\mbox{.}\]
\[\sigma^{[2]}(n)=\ln^{[2]}{(n)}\mbox{.}\]
\[\sigma^{[j]}(n)=\ln^{[j]}{(n)}\prod_{i=2}^{j-1}\sqrt{\ln^{[i]}(n)}\mbox{, where }j\geq 3\mbox{.}\]
\[\rho^{[j]}(n)=\sigma^{[j]}(n)\frac{\sqrt{\ln{n}}}{\sqrt{n}}\mbox{, where }j\in\mathbb{N}\mbox{.}\]
\end{definition}

In the current article we improve the upper bound (\ref{dyhc67vzcxk}) for the number of privileged words. We prove the following theorem.
\begin{theorem}
\label{jhd8kjenj5d4g}
If $q>1$ is a size of the alphabet and $j\in\mathbb{N}$ then there are constants $\alpha_j$ and $n_j$ such that \[\priv(n)\leq \alpha_j\rho^{[j]}(n)q^{n}\mbox{, where }n\geq n_j\mbox{.}\]
\end{theorem}
\begin{remark}
Note in Theorem \ref{jhd8kjenj5d4g} that the constants $\alpha_j$ and $n_j$ depend on the size of the alphabet $q$ and on the constant $j$. 
\end{remark}
\begin{remark}
It is easy to verify that $\lim_{n\rightarrow\infty}\frac{\rho^{[j]}(n)q^n}{\rho^{[j+1]}(n)q^n}=\infty$ for every positive integer $j$. It means that the bigger $j$ the better upper bound $\alpha_{j}\rho^{[j]}(n)q^n$.
\end{remark}
\begin{example}
We have that \begin{itemize}\item $\rho^{[1]}(n)q^n=n^{-\frac{1}{2}}q^n\ln{n}$,
\item $\rho^{[2]}(n)q^n=n^{-\frac{1}{2}}q^n\sqrt{\ln{n}}\ln{\ln{n}}$,
\item $\rho^{[3]}(n)q^n=n^{-\frac{1}{2}}q^n\sqrt{\ln{n}}(\ln{\ln{\ln{n}}})\sqrt{\ln{\ln{n}}}$, and 
\item $\rho^{[4]}(n)q^n=n^{-\frac{1}{2}}q^n\sqrt{\ln{n}}(\ln{\ln{\ln{\ln{n}}}})\sqrt{\ln{\ln{\ln{n}}}}\sqrt{\ln{\ln{n}}}$.
\end{itemize}
\end{example}

To prove our result, we apply in principle the same ideas like in \cite{RUKAVICKA2020105917}. It means that we enumerate the privileged words depending on the length of the maximal border. We distinguish ``short'' and ``long'' borders. It turns out that the number of privileged words with a short border is bigger than the number of privileged words with a long border. When comparing the proof in the current article and the proof in  \cite{RUKAVICKA2020105917}, the essential difference is that we consider only privileged words instead of all words when enumerating the borders. Due to this difference Theorem \ref{jhd8kjenj5d4g} does not hold for closed words; recall that the upper bound (\ref{dyhc67vzcxk}) holds for closed words. To facilitate the comprehension we use mostly the same notation like in \cite{RUKAVICKA2020105917}.

\section{Preliminaries}
Let $\Alphabet$ be an alphabet with $q$ letters, where $q>1$. Let $\epsilon$ denote the empty word. Let $\Alphabet^m$ denote the set of all words of length $m$, let $\Alphabet^+=\bigcup_{m\geq 1}\Alphabet^m$, and let $\Alphabet^*=\Alphabet^+\cup\{\epsilon\}$. We have that $\Alphabet^0=\{\epsilon\}$ and that $\vert \Alphabet^m\vert=q^m$. 

Let $\Alphabet_w(n)$ denote the number of words of length $n$ that do not contain the factor $w\in \Alphabet^*$. 
Let \[\mu(n,m)=\max\{\Alphabet_w(n)\mid w\in \Alphabet^m\}\mbox{.}\] The function $\mu(n,m)$ represents the maximal value of $\Alphabet_w(n)$ for all $w$ of length $m$.

Let $\widehat \priv(n)\subseteq\Alphabet^*$ denote the set of all privileged words of length $n$ and let $\widehat \priv(n,m)$ denote the set of all privileged words of length $n$ having the maximal border of length $m$. Let $\priv(n)=\vert \widehat \priv(n)\vert$ and $\priv(n,m)=\vert \widehat \priv(n,m)\vert$.

\begin{remark}
Note that:
\begin{itemize}
\item
$\widehat\priv(n)=\bigcup_{m=1}^{n-1}\widehat\priv(n,m)$.
\item $\widehat\priv(n,m)\cap\widehat\priv(n,\overline m)=\emptyset$, where $m\not=\overline m$. 
\end{itemize}
\end{remark} 

Let $\omega(n)=\frac{1}{\ln{q}}(\ln{n}-\ln{\ln{n}})\in\mathbb{R}$, where $n\in\mathbb{N}$.

Let $\Pi$ denote the set of all functions $\pi(n):\mathbb{N}\rightarrow \mathbb{N}$ such that $\pi(n)\in \Pi$ if and only if $1\leq\pi(n)\leq \max\{1,\omega(n)\}$ and $\pi(n)\leq \pi(n+1)$ for all $n\in \mathbb{N}$. We apply the function $\max$, because $\omega(n)<1$ for some small $n$.

\section{Previous results}
In this section we recall the results from  \cite{RUKAVICKA2020105917} that we will need for the current article.

In \cite{RUKAVICKA2020105917}, an upper bound for $\mu(n,m)$ was shown; it means an upper bound for the number of words of length $n$ that avoid some factor of length $m$.
\begin{lemma}(\cite[Lemma $2.1$]{RUKAVICKA2020105917})
\label{ujd6562v3b6}
If $n,m\in \mathbb{N}$ then
\[\mu(n,m)\leq q^n\left(1-\frac{1}{q^m}\right)^{\lfloor\frac{n}{m}\rfloor}\mbox{.}\]
\end{lemma}

Let $\beta=\frac{1}{\ln{q}}\in\mathbb{R}$. Using Lemma \ref{ujd6562v3b6} it was shown that the number of words of length $n$ avoiding some factor of length shorter than $\pi(n)\in \Pi$ grows with $n$ approximately as the number of all words of length $n-\beta \ln{n}$. 
\begin{theorem}(\cite[Theorem $2.3$]{RUKAVICKA2020105917})
\label{nk211d2qp6h}
If $\pi(n)\in \Pi$ then there is a constant $\cca\in \mathbb{R}$ such that for all $n\in \mathbb{N}$ we have that
\[\frac{\mu(n,\pi(n))}{ q^{n-\beta\ln{n}}}\leq \cca\mbox{.}\]
\end{theorem}

Let $h(n)=\lfloor\beta\ln{n}\rfloor$. 
We present Theorem \ref{nk211d2qp6h} in a more useful form for our next proofs.
\begin{corollary}(\cite[Corollary $2.4$]{RUKAVICKA2020105917})
\label{ff21a2w12ppmnw5}
If $\overline \pi(n)\in \Pi$ then there is a constant $\ccb\in \mathbb{R}$ such that for all $n\in \mathbb{N}$ we have that
\[\frac{\mu(n-2\overline \pi(n),\overline \pi(n))}{q^{n-h(n)}}\leq \ccb\mbox{.}\]
\end{corollary}

Let $\kappa>1$ be a real constant and let \begin{equation}\label{nndjf99ej}\overline h(n)=\max\{1,\lfloor\frac{1}{\kappa}\omega(n)\rfloor\}\mbox{.}\end{equation} We apply the function $\max$ to assert that $\overline h(n)\geq 1$ for all $n$. 
\begin{remark}
In our proofs, the privileged words will be enumerated depending on the length of the maximal border; recall the ``short'' and ``long'' borders mentioned in the introduction. We will distinguish the length of the maximal border for $m<\overline h(n)$ and for $m\geq\overline h(n)$. 
\end{remark}

The next technical lemma shows an upper bound for $q^{-h(n)+\overline h(n)}$.
\begin{lemma}(\cite[Lemma $2.7$]{RUKAVICKA2020105917}) 
\label{pplk856t59e9e}
There is a constant $\ccc\in \mathbb{R}$ such that for all $n\in \mathbb{N}$ we have that 
\[q^{-h(n)+\overline h(n)}\leq \ccc q^{\frac{\ln{n}}{\ln{q}}\left(\frac{1}{\kappa}-1\right)}\mbox{.}\]
\end{lemma}

\section{Upper bound for privileged words}

We show that if $w$ is the maximal border of a privileged word $u$ then $w$ is privileged.
\begin{lemma}
\label{u8duyfh9}
If $n,m\in\mathbb{N}$, $u\in\widehat\priv(n,m)$, $w\in\Theta(u)$, and $\vert w\vert=m$ then $w\in\widehat\priv(m)$.
\end{lemma}
\begin{proof}
Suppose that $w$ is not privileged. Then since $u$ is privileged there is a privileged border $\overline w\in\Theta(u)\setminus\{w\}$ with exactly two occurrences in $u$. We have that $\vert \overline w\vert<\vert w\vert$, $\overline w$ is a prefix of $w$, and $\overline w$ is a suffix of $u$. Since $w$ has at least two occurrences in $u$, it follows that $\overline w$ has at least three occurrences in $w$. This is a contradiction. We conclude that $w$ is privileged word. This ends the proof.
\end{proof}

We show a recursive upper bound for the number of privileged words $\priv(n,m)$. The following lemma is a variation of \cite[Lemma $2.5$]{RUKAVICKA2020105917} for privileged words.
\begin{lemma} \label{dhdud77eghej} Suppose $n,m\in \mathbb{N}$. We have that
\begin{itemize}
  \item
  If $2m>n$ then $\priv(n,m)\leq q^{\lceil\frac{n}{2}\rceil}$.
  \item
  If $2m\leq n$ then $\priv(n,m)\leq \priv(m)\mu(n-2m,m)$. 
\end{itemize}
\end{lemma}
\begin{proof}
If $2m>n$ and $w\in \Alphabet^m$ then there is obviously at most one word $u$ with $\vert u\vert=n$ having a prefix and a suffix $w$; realize that the prefix $w$ and the suffix $w$ would overlap with each other. If such $u$ exists then the first half of $u$ uniquely determines the second half of $u$. If follows that \begin{equation}\label{djcdkk33jd}\priv(n,m)\leq q^{\lceil\frac{n}{2}\rceil}\mbox{.}\end{equation}

Let $\Factor(v)$ denote the set of all factors of $v\in \Alphabet^*$. If $n\geq 2m$ then let \[T(n,m)=\{wuw\mid u\in \Alphabet^{n-2m}\mbox{ and }w\in \widehat\priv(m)\mbox{ and }w\not\in \Factor(u)\}\mbox{.}\] If $n\geq 2m$ then Lemma \ref{u8duyfh9} implies that \begin{equation}\label{jdid8bbjd4r}\widehat \priv(n,m)\subseteq T(n,m)\mbox{.}\end{equation} It is easy to see that \begin{equation}\label{msd98dhhwiu}\vert T(n,m)\vert\leq\priv(m)\mu(n-2m,m)\mbox{.}\end{equation}
The lemma follows from (\ref{djcdkk33jd}), (\ref{jdid8bbjd4r}), and (\ref{msd98dhhwiu}). This ends the proof.
\end{proof}

\begin{definition}
\label{gbd7df99hhdg6} 
Let $\up$ be a set of functions $\rho:\mathbb{N}\rightarrow\mathbb{R}$ such that $\rho\in \up$ if and only if there is $n_0\in\mathbb{N}$ such that for all $n>n_0$ we have that \begin{enumerate}\item $q^{n}\rho(n)\geq \priv(n)\mbox{,}$ \item $\rho(n)\geq\rho(n+1)$, and 
\item \label{uyd7bvxgdu} $q^{n}\rho(n)\leq q^{n+1}\rho(n+1)$.\end{enumerate} Thus $\up(X)$ is a set of non-increasing functions $\rho$ such that $q^{n}\rho(n)$ are upper bounds for the number of privileged words.
\end{definition}

Using the functions of $\up$, we can restate Lemma \ref{dhdud77eghej} as follows. We omit the proof, as it follows immediately from Lemma \ref{dhdud77eghej} and Definition \ref{gbd7df99hhdg6}.
\begin{lemma} \label{dyduf78jehdf} If $\rho\in\up$ then there is $n_0\in\mathbb{N}$ such that for all $n,m\in \mathbb{N}$ with $n>n_0$ we have that
\begin{itemize}
  \item
  If $2m>n$ then $\priv(n,m)\leq q^{\lceil\frac{n}{2}\rceil}$.
  \item
  If $2m\leq n$ then $\priv(n,m)\leq q^{m}\rho(m)\mu(n-2m,m)$. 
\end{itemize}
\end{lemma}

We show an upper bound for the number of privileged words of length $n$ having the maximal border shorter than $\overline h(n)$.
\begin{lemma}
\label{bv56shdi883j} If $\rho\in\up$ then there are constants $\ccd\in\mathbb{R}$ and $n_0\in\mathbb{N}$ such that for all $n>n_0$ we have that
\[\sum_{m=1}^{\overline h(n)-1}\priv(n,m)\leq \ccd\frac{\ln{n}}{\kappa}q^{n}n^{\frac{1}{\kappa}-1}{\rho(\overline h(n))}\mbox{.}\]
\end{lemma}
\begin{proof} 
From (\ref{nndjf99ej}) it follows 
that $\overline h(n)< \frac{n}{2}$ for sufficiently large $n$. Hence Lemma \ref{dhdud77eghej} implies that 
\begin{equation}\label{duf8fiedhjf}\sum_{m=1}^{\overline h(n)-1}\priv(n,m)\leq\sum_{m=1}^{\overline h(n)-1}\priv(m)\mu(n-2m,m)\mbox{.}\end{equation}
From (\ref{nndjf99ej}) it follows that \begin{equation}\label{dhnjv8f78dd}\overline h(n)\leq \frac{\ln{n}}{\kappa\ln{q}}\mbox{ for sufficiently large $n$.}\end{equation}
Corollary \ref{ff21a2w12ppmnw5} implies that $\mu(n-2m,m)\leq \ccb q^{n-h(n)}$, where $1\leq m\leq\overline h(n)\leq h(n)$ and $\ccb\in \mathbb{R}$ is some constant.  Then 
it follows from Lemma \ref{pplk856t59e9e}, Property \ref{uyd7bvxgdu} of Definition \ref{gbd7df99hhdg6}, and (\ref{dhnjv8f78dd}) that 
\begin{equation}\label{kifu66969t6}
\begin{split}\sum_{m=1}^{\overline h(n)-1}\priv(m)\mu(n-2m,m) &\leq \sum_{m=1}^{\overline h(n)}q^{m}\rho(m) \ccb q^{n-h(n)}\\ &\leq \overline h(n)q^{\overline h(n)}{\rho(\overline h(n))}\ccb q^{n-h(n)}
\\&\leq \ccb\ccc\overline h(n)q^{n+\frac{\ln{n}}{\ln{q}}(\frac{1}{\kappa}-1)}{\rho(\overline h(n))}\\&=\ccb\ccc\frac{\ln{n}}{\kappa\ln{q}}q^{n+\frac{\ln{n}}{\ln{q}}(\frac{1}{\kappa}-1)}{\rho(\overline h(n))}\\&= \ccb\ccc\frac{\ln{n}}{\kappa\ln{q}}q^{n}n^{\frac{1}{\kappa}-1}{\rho(\overline h(n))}\mbox{.}
\end{split}
\end{equation}
Let $\ccd=\frac{\ccb\ccc}{\ln{q}}$. The lemma follows from (\ref{duf8fiedhjf}) and (\ref{kifu66969t6}). This ends the proof.
\end{proof}

We show an upper bound for the number of privileged words of length $n$ with the maximal border longer than $\overline h(n)$ and shorter than $\frac{n}{2}$.
\begin{lemma}
\label{dyu7e8vczg6t}There are constants $\cce\in\mathbb{R}$ and $n_0\in\mathbb{N}$ such that  
\[\sum_{m=\overline h(n)}^{\lceil\frac{n}{2}\rceil}\priv(n,m) \leq \cce q^{n}(\ln{n})^{\frac{1}{\kappa}}n^{-\frac{1}{\kappa}}\mbox{, where }n>n_0\mbox{.}\]
\end{lemma}
\begin{proof}
From Lemma \ref{dhdud77eghej} and from $\mu(n-2m,m)\leq q^{n-2m}$ we have for sufficiently large $n$ that  
\begin{equation}\label{bndhd65shdw}\sum_{m=\overline h(n)}^{\lceil\frac{n}{2}\rceil}\priv(n,m)\leq\sum_{m=\overline h(n)}^{\lceil\frac{n}{2}\rceil}\priv(m)\mu(n-2m,m)\leq \sum_{m=\overline h(n)}^{\lceil\frac{n}{2}\rceil}\priv(m)q^{n-2m}\mbox{.}\end{equation}
From (\ref{nndjf99ej}) we have for sufficiently large $n$ that \begin{equation}\label{rrjtuk66585g65}\begin{split}q^{-\overline h(n)} &\leq q^{-\frac{1}{\kappa \ln{q}}(\ln{n}-\ln{\ln{n}})+1} \\ &= q(\ln{n})^{\frac{1}{\kappa}}q^{-\frac{1}{\kappa \ln{q}}\ln{n}} \\ &= q(\ln{n})^{\frac{1}{\kappa}}n^{-\frac{1}{\kappa}}\mbox{.}\end{split}\end{equation} From (\ref{rrjtuk66585g65}) it follows that 
\begin{equation}
\label{rrth2111h25y}
\begin{split}
\sum_{m=\overline h(n)}^{\lceil\frac{n}{2}\rceil}\priv(m)q^{n-2m}&\leq\sum_{m=\overline h(n)}^{\lceil\frac{n}{2}\rceil}q^mq^{n-2m} \\&\leq q^n\sum_{m=\overline h(n)}^{\lceil\frac{n}{2}\rceil}q^{-m}\\&=q^n\frac{1-q^{-(\lceil\frac{n}{2}\rceil+1)}}{1-q^{-1}}-q^n\frac{1-q^{-\overline h(n)}}{1-q^{-1}}\\ &\leq \frac{q^{n-\overline h(n)}}{1-q^{-1}}\\&\leq \frac{q^{n+1}(\ln{n})^{\frac{1}{\kappa}}n^{-\frac{1}{\kappa}}}{1-q^{-1}}\mbox{.}
\end{split}
\end{equation}

Let $\cce=\frac{q}{1-q^{-1}}$. Then the lemma follows from (\ref{bndhd65shdw}) and (\ref{rrth2111h25y}). This ends the proof.
\end{proof}

We show an approximation for the function $\frac{\sigma^{[j]}(\overline h(n))}{\sigma^{[j+1]}(n)}$ as $n$ tends to infinity.
\begin{lemma}
\label{f88hgye9fdj}
If $j\in\mathbb{N}$ then \[\begin{split}\lim_{n\rightarrow\infty}\frac{\sigma^{[j]}(\overline h(n))\sqrt{\ln{(\overline h(n))}}}{\sigma^{[j+1]}(n)}=1\mbox{.}\end{split}\]
\end{lemma}
\begin{proof}

It follows from (\ref{nndjf99ej}) that for sufficiently large $n$ we have that $\overline h(n)=\lfloor\frac{1}{\kappa\ln{q}}(\ln{n}-\ln{\ln{n}})\rfloor$.  Then it is easy to see that 
\begin{equation}\label{mn87djdf99sj}\begin{split}\lim_{n\rightarrow\infty}\frac{\ln^{[j]}{(\overline h(n))}}{\ln^{[j+1]}{(n)}}=\lim_{n\rightarrow\infty}\frac{\ln^{[j]}{(\lfloor\frac{1}{\kappa\ln{q}}(\ln{n}-\ln{\ln{n}})\rfloor)}}{\ln^{[j+1]}{(n)}}=1\mbox{.}\end{split}\end{equation}
Let \[y(n)=\frac{\sigma^{[j]}(\overline h(n))\sqrt{\ln{(\overline h(n))}}}{\sigma^{[j+1]}(n)}\mbox{}\]
and let 
\[\overline y(n)=\frac{\sqrt{\ln{(\overline h(n))}}\prod_{i=2}^{j-1}\sqrt{\ln^{[i]}(\overline h(n))}}{\prod_{i=2}^{j}\sqrt{\ln^{[i]}(n)}}\mbox{.}\]

From (\ref{mn87djdf99sj}) it follows that 
\begin{equation}\label{b789sj22jdh}\begin{split}\lim_{n\rightarrow\infty}\overline y(n)&=\lim_{n\rightarrow\infty}\frac{\sqrt{\ln{(\overline h(n))}}\prod_{i=2}^{j-1}\sqrt{\ln^{[i]}(\overline h(n))}}{\prod_{i=2}^{j}\sqrt{\ln^{[i]}(n)}}\\&=\lim_{n\rightarrow\infty}\frac{\sqrt{\ln{(\overline h(n))}}\sqrt{\ln^{[2]}(\overline h(n))}\sqrt{\ln^{[3]}(\overline h(n))}\cdots\sqrt{\ln^{[j-1]}(\overline h(n))}}{\sqrt{\ln^{[2]}(n)}\sqrt{\ln^{[3]}(n)}\cdots \sqrt{\ln^{[j-1]}(n)}\sqrt{\ln^{[j]}(n)}}\\&=1\mbox{.}\end{split}\end{equation}

From (\ref{mn87djdf99sj}) and (\ref{b789sj22jdh}) it follows that 
\begin{equation}\label{hgst78f7erjhc}\begin{split}\lim_{n\rightarrow\infty}y(n)&=\lim_{n\rightarrow\infty}\frac{\ln^{[j]}{(\overline h(n))}\left(\prod_{i=2}^{j-1}\sqrt{\ln^{[i]}(\overline h(n))}\right)\sqrt{\ln{(\overline h(n))}}}{\ln^{[j+1]}{(n)}\prod_{i=2}^{j}\sqrt{\ln^{[i]}(n)}}=1\mbox{.}\end{split}\end{equation}

From (\ref{mn87djdf99sj}) it follows that 
\begin{equation}\label{b6cgsdfd8jj}\lim_{n\rightarrow\infty}\frac{\sigma^{[1]}(\overline h(n))\sqrt{\ln{(\overline h(n))}}}{\sigma^{[2]}(n)}=\lim_{n\rightarrow\infty}\frac{\sqrt{\ln{(\overline h(n))}}\sqrt{\ln{(\overline h(n))}}}{\ln^{[2]}{(n)}}=1\mbox{}\end{equation}
and
\begin{equation}\label{bii89dbvhdh}\lim_{n\rightarrow\infty}\frac{\sigma^{[2]}(\overline h(n))\sqrt{\ln{(\overline h(n))}}}{\sigma^{[3]}(n)}=\lim_{n\rightarrow\infty}\frac{\ln^{[2]}{(\overline h(n))}\sqrt{\ln{(\overline h(n))}}}{\ln^{[3]}{(n)}\sqrt{\ln^{[2]}(n)}}=1\mbox{.}\end{equation}
The lemma follows from (\ref{hgst78f7erjhc}), (\ref{b6cgsdfd8jj}), and (\ref{bii89dbvhdh}).
This ends the proof.
\end{proof}
The next technical lemma will be used in the proof of Theorem \ref{thuut7678yyhty}.
\begin{lemma}
\label{ud8d7vbvu}
If $j\in\mathbb{N}$ then there are constants $\ccf$ and $n_0$ such that for all $n\in\mathbb{N}$ with $n>n_0$ we have that \[\ccf\sigma^{[j+1]}{(n)}\sqrt{\ln{n}}q^{n}n^{-\frac{1}{\ktwo}}\geq \max\left\{\ln{(n)}q^{n}n^{-\frac{1}{\ktwo}}\rho^{[j]}(\overline h(n)), \sqrt{\ln{n}}q^{n}n^{-\frac{1}{\ktwo}}\right\}\mbox{.}\]
\end{lemma}
\begin{proof}
Let \[\begin{split}y(n) &= \frac{\ln{(n)}q^{n}n^{-\frac{1}{\ktwo}}\rho^{[i]}(\overline h(n))}{\sigma^{[i+1]}(n)\sqrt{\ln{n}}q^{n}n^{-\frac{1}{\ktwo}}}\\ &=\frac{\sqrt{\ln{n}}\rho^{[i]}(\overline h(n))}{\sigma^{[i+1]}(n)}\\&=\frac{\sqrt{\ln{n}}\sigma^{[i]}(\overline h(n))\sqrt{\ln{(\overline h(n))}}}{\sigma^{[i+1]}(n)\sqrt{\overline h(n)}}\mbox{.}\end{split}\]
Lemma \ref{f88hgye9fdj} and (\ref{nndjf99ej}) imply that
\begin{equation}
\label{vvbdhd87f}
\begin{split}
\lim_{n\rightarrow\infty}y(n) &= 
\lim_{n\rightarrow\infty}\frac{\sqrt{\ln{n}}}{\sqrt{\overline h(n)}}\\&=
\lim_{n\rightarrow\infty}\frac{\sqrt{\ln{n}}}{\sqrt{\lfloor\frac{1}{\kappa\ln{q}}(\ln{n}-\ln{\ln{n}})\rfloor}}\\&=
\sqrt{\kappa\ln{q}}
\mbox{.}
\end{split}
\end{equation}
Obviously $\sigma^{[j+1]}{(n)}\sqrt{\ln{n}}q^{n}n^{-\frac{1}{\ktwo}}\geq \sqrt{\ln{n}}q^{n}n^{-\frac{1}{\ktwo}}$ for sufficiently large $n$. Thus 
 the lemma follows from (\ref{vvbdhd87f}).
This ends the proof.
\end{proof}

We show how to improve the upper bound for the number of privileged words on condition that there is already an upper bound of the form $\alpha_j\rho^{[j]}(n)q^n$, where $\alpha_j$ is a constant.
\begin{theorem}
\label{thuut7678yyhty} 
If there is $j\in\mathbb{N}$ and a constant $\alpha_j$ such that $\alpha_j\rho^{[j]}\in\up$ then there is a constant $\alpha_{j+1}$ such that 
$\alpha_{j+1}\rho^{[j+1]}\in\up$.
\end{theorem}
\begin{proof}
Obviously we have that 
\begin{equation}\label{t12976ukd2f5g} \begin{split}
\priv(n)=\sum_{m=1}^{n-1}\priv(n,m) &=
\sum_{m=1}^{\lceil\frac{n}{2}\rceil}\priv(n,m)+\sum_{m=\lceil\frac{n}{2}\rceil+1}^{n-1} \priv(n,m)\mbox{.}\end{split}
\end{equation}
From Lemma \ref{dyduf78jehdf} we get that 
\begin{equation}
\label{hdydue88hjj}
\sum_{m=\lceil\frac{n}{2}\rceil+1}^{n-1} \priv(n,m)\leq \sum_{m=\lceil\frac{n}{2}\rceil+1}^{n-1}q^{\lceil\frac{n}{2}\rceil}\leq \frac{n}{2}q^{\lceil\frac{n}{2}\rceil} \mbox{.}
\end{equation}

From Lemma \ref{bv56shdi883j} and Lemma \ref{dyu7e8vczg6t}, it follows that there are constants $\ccd, \cce\in\mathbb{R}$ and $n_0\in\mathbb{N}$ such that for all $n> n_0$ we have that
\begin{equation}\label{dujehdkjsrt5} \begin{split}
\sum_{m=1}^{\lceil\frac{n}{2}\rceil}\priv(n,m)&= \sum_{m=1}^{\overline h(n)-1}\priv(n,m)+\sum_{m=\overline h(n)}^{\lceil\frac{n}{2}\rceil}\priv(n,m)\\ &\leq  \ccd\frac{\ln{n}}{\kappa}q^{n}n^{\frac{1}{\kappa}-1}\alpha_j\rho^{[j]}(\overline h(n))+\cce(\ln{n})^{\frac{1}{\kappa}}q^{n}n^{-\frac{1}{\kappa}}\mbox{.}\end{split}
\end{equation}
Let $\kappa=2$. From Lemma \ref{ud8d7vbvu} and (\ref{dujehdkjsrt5}), it follows that there are constants $\ccf\in\mathbb{R}$ and $n_0\in\mathbb{N}$ such that for all $n>n_0$ we have that 
\begin{equation}\label{bdd9ifjeij} 
\sum_{m=1}^{\lceil\frac{n}{2}\rceil}\priv(n,m)\leq  2\ccf\sigma^{[j+1]}{(n)}\sqrt{\ln{n}}q^{n}n^{-\frac{1}{\ktwo}}\mbox{.}
\end{equation}
From (\ref{t12976ukd2f5g}), (\ref{hdydue88hjj}), and (\ref{bdd9ifjeij}) it follows that 
\begin{equation}\label{dru7e8jbbvs}
\priv(n)\leq 2\ccf\sigma^{[j+1]}{(n)}\sqrt{\ln{n}}q^{n}n^{-\frac{1}{\ktwo}} + \frac{n}{2}q^{\lceil\frac{n}{2}\rceil}
\mbox{.}
\end{equation}
We have that \begin{equation}\label{dhyyd7vb2ws}\lim_{n\rightarrow\infty}\frac{nq^{\frac{n}{2}}}{q^{n}n^{-\frac{1}{\ktwo}}}=0\mbox{.}\end{equation}
From (\ref{dru7e8jbbvs}) and (\ref{dhyyd7vb2ws}) it follows that there is a constant $\alpha_{j+1}$ such that 
\[\priv(n)\leq  \alpha_{i+1}\sigma^{[j+1]}{(n)}\sqrt{\frac{\ln{n}}{n}}q^{n}=\alpha_{i+1}\rho^{[j+1]}(n)q^n\mbox{.}\]
This completes the proof.
\end{proof}

Now we can step to the proof of the main theorem of the article.
\begin{proof}(Proof of Theorem \ref{jhd8kjenj5d4g})
From (\ref{dyhc67vzcxk}) it follows that there is a constant $\alpha_1$ such that $\alpha_1\rho^{[1]}(n)\in\up$. Hence Theorem \ref{thuut7678yyhty} implies that for every $j\in\mathbb{N}$ there is a constant $\alpha_j$ such that $\alpha_j\rho^{[j]}(n)\in\up$.
This ends the proof.
\end{proof}

\section*{Acknowledgments}
The author acknowledges support by the Czech Science
Foundation grant GA\v CR 13-03538S and by the Grant Agency of the Czech Technical University in Prague, grant No. SGS14/205/OHK4/3T/14.

\bibliographystyle{siam}
\IfFileExists{biblio.bib}{\bibliography{biblio}}{\bibliography{../!bibliography/biblio}}

\end{document}